\documentclass{amsart}

\usepackage{fancyhdr}

\usepackage{amsmath,amssymb,amsthm,mathrsfs}

\usepackage[colorlinks=true,linkcolor=blue]{hyperref}
\makeatletter
\renewcommand*{\eqref}[1]{%
  \hyperref[{#1}]{\textup{\tagform@{\ref*{#1}}}}%
}
\makeatother

\usepackage{enumitem,natbib,nicefrac}

\setcitestyle{numbers,open={[},close={]}}

\numberwithin{equation}{section}

\def\e{{\rm e}}
\def\<{{\langle}}
\def\>{{\rangle}}
\def\R{{\mathbb R}}
\def\T{{\mathbb T}}

\def\d{{\rm d}}
\def\ri{{\rm i}}
\def\Z{{\mathbb Z}}

\def\:{\colon }

\def\N{{\mathbb N}}
\def\dist{{\rm dist}}
\def\A{{\mathcal A}}

\def\be#1{\begin{equation}\label{#1}}
\def\ee{\end{equation}}
\def\dual#1#2{\langle#1,#2\rangle}
\def\eps{\varepsilon}
\def\D{{\mathcal D}}
\def\dv{{\rm div}\,}
\def\Lp{{\mathcal L}}

\def\sbsection#1{\subsection{#1}$ $\smallskip}

\newcommand{\lra}[1]{\left\lbrack{#1}\right\rbrack}
\newcommand{\lrb}[1]{\left\lbrace{#1}\right\rbrace}

\newtheorem{theorem}{Theorem}[section]

\newtheorem{corollary}[theorem]{Corollary}
\newtheorem{lemma}[theorem]{Lemma}
\newtheorem{proposition}[theorem]{Proposition}

\newtheorem{definition}[theorem]{Definition}

\providecommand{\abs}[1]{\left\lvert{#1}\right\rvert}
\providecommand{\norm}[1]{\left\lVert{#1}\right\rVert}

\begin{document}


\title{Simultaneous approximation in Lebesgue and Sobolev norms via eigenspaces}

\author{Charles L. Fefferman}
\address{Department of Mathematics, Princeton University, Fine Hall,
Washington Road, Princeton, NJ 08544}
\email{\href{mailto:cf@math.princeton.edu}{cf@math.princeton.edu}}
\thanks{CLF was supported in part by NSF grant DMS 1608782.}

\author{Karol W. Hajduk}
\address{Department of Mathematics and Statistics, Faculty of Science, Masaryk University,
Building 08, Kotl\'a{\v r}sk\'a 2, 611 37, Brno, Czech Republic}
\email{\href{mailto:hajduk@math.muni.cz}{hajduk@math.muni.cz}}
\thanks{KWH was supported by an EPSRC Standard DTG EP/M506679/1 and by the Warwick Mathematics Institute.}

\author{James C. Robinson}
\address{Mathematics Institute, Zeeman Building, University of Warwick,
Coventry, CV4 7AL, United Kingdom}
\email{\href{mailto:j.c.robinson@warwick.ac.uk}{j.c.robinson@warwick.ac.uk}}


\date{August 3, 2021.}


\begin{abstract}
  We approximate functions defined on smooth bounded domains by elements of the eigenspaces of the Laplacian or the Stokes operator in such a way that the approximations are bounded and converge in both Sobolev and Lebesgue spaces. We prove an abstract result referred to fractional power spaces of positive, self-adjoint, compact-inverse operators on Hilbert spaces, and then obtain our main result by using the explicit form of these fractional power spaces for the Dirichlet Laplacian and Stokes operators. As a simple application, we prove that all weak solutions of the incompressible convective Brinkman--Forchheimer equations posed on a bounded domain in ${\mathbb R}^3$ satisfy the energy equality.
\end{abstract}

\maketitle

\section{Introduction}

In this paper we describe a method that allows one to use truncated (but weighted) eigenfunction expansions in order to obtain smooth approximations of functions defined on bounded domains in a way that behaves well with respect to both Lebesgue spaces and (primarily $L^2$-based) Sobolev spaces, and that also respects the `side conditions' that often occur in boundary value problems (e.g.\ Dirichlet boundary data or a divergence-free condition).

If $u\in L^2(\T^d)$ with
\begin{equation}\label{FSu}
u=\sum_{k\in\Z^d}\hat u_k\e^{\ri k\cdot x}
\end{equation}
and we set
$$
u_n:=\sum_{k\in\Z^d:\ |k|\le n}\hat u_k\e^{\ri k\cdot x},
$$
where $|k|$ is the Euclidean length of $k$, then this truncation behaves well in $L^2$-based spaces:
$$
\|u_n-u\|_X\to0\qquad\mbox{and}\qquad \|u_n\|_X\le \|u\|_X
$$
for $X=L^2(\T^d)$ or $H^s(\T^d)$.

However, the same is not true in $L^p(\T^d)$ for $p\neq 2$ if $d\ne 1$: there \textit{is no constant $C$} such that
$$
\|u_n\|_{L^p}\le C\|u\|_{L^p}\qquad\mbox{for every }u\in L^p(\T^3).
$$
This follows from the result of Fefferman \cite{CLF} concerning the ball multiplier for the Fourier transform; standard `transference' results (see Grafakos \cite{Grafakos}, for example) then yield the result for Fourier series. There are similar problems when using eigenfunction expansions in bounded domains, see Babenko \cite{Babenko}.

In the periodic setting these problems can be overcome by considering the component-wise truncation over `cubes' rather than `spheres' of Fourier modes. If for $u$ as in \eqref{FSu} we define
$$
u_{[n]}:=\sum_{|k_j|\le n}\hat u_k\e^{\ri k\cdot x},\qquad\mbox{where }k=(k_1,\ldots,k_d),
$$
then it follows from good properties of the truncation in 1D and the product structure of the Fourier expansion that
$$
\|u_{[n]}-u\|_{L^p}\to0\qquad\mbox{and}\qquad\|u_{[n]}\|_{L^p}\le C_p\|u\|_{L^p}\quad u\in L^p(\T^d)
$$
(see Muscalu \& Schlag \cite{Muscalu-Schlag}, for example). Hajduk \& Robinson \cite{HR} used this approach to prove that all weak solutions of the convective Brinkman--Forchheimer (CBF) equations
\be{CBF-intro}
\partial_tu-\Delta u+(u\cdot\nabla)u+|u|^2u+\nabla p=0,\qquad\nabla\cdot u=0
\ee
 on $\T^3$ satisfy the energy equality (for more details see Section \ref{sec-CBF}).

There is no known corresponding `good' selection of eigenfunctions in bounded domains that will produce truncations that are bounded in $L^p$. To circumvent this we suggest two possible approximation schemes in this paper: for one scheme we use the linear semigroup arising from an appropriate differential operator (the Laplacian or Stokes operator); for the second we combine this with a truncated eigenfunction expansion.

We discuss these methods in the abstract setting of fractional power spaces (i.e.\ the domains of fractional powers of some linear operator) in Section \ref{approx-abstract}. In Section \ref{explicit} we recall the explicit form of these fractional power spaces for the Dirichlet Laplacian and Stokes operators, and derive some additional properties required in what follows. We combine these two sections to give our appoximation theorems in Section \ref{ourapprox}, and then use our eigenspace-approximation method to prove the validity of the energy equality for weak solutions of the CBF equations \eqref{CBF-intro} on bounded domains in Section \ref{sec-CBF}.

\section{Approximation in fractional power spaces}\label{approx-abstract}

We want to investigate simultaneous approximation in fractional power spaces and a second space $\Lp$, which in our applications will be one of the spaces $L^p(\Omega)$ [potentially with side conditions when treating divergence-free vector-valued functions].

\sbsection{Fractional power spaces}\label{sb-XpDr}

We suppose that $H$ is a separable Hilbert space, with inner product $\<\cdot,\cdot\>$ and norm $\|\cdot\|$, and that $A$ is a positive, self-adjoint operator on $H$ with compact inverse. In this case $A$ has a complete set of orthonormal eigenfunctions $\{w_n\}$ with corresponding eigenvalues $\lambda_n>0$, which we order so that $\lambda_{n+1}\ge\lambda_n$.

Recall that for any $\alpha\ge0$ we can define $D(A^\alpha)$ as the subspace of $H$ where
\begin{equation}\label{As-def}
D(A^\alpha) :=\left\{u=\sum_{j=1}^\infty \hat u_jw_j:\ \sum_{j=1}^\infty\lambda_j^{2\alpha}|\hat u_j|^2<\infty\right\}.
\end{equation}
For $\alpha<0$ we can take this space to be the dual of $D(A^{-\alpha})$; the expression in \eqref{As-def} can then be understood as an element in the completion of the space of finite sums with respect to the $D(A^\alpha)$ norm defined below in \eqref{DAnorm}. For all $\alpha\in\R$ the space $D(A^\alpha)$ is a Hilbert space with inner product
$$
\<u,v\>_{D(A^\alpha)}:=\sum_{j=1}^\infty\lambda_j^{2\alpha}\hat u_j\hat v_j
$$
and corresponding norm
\be{DAnorm}
\|u\|_{D(A^\alpha)}^2 :=\sum_{j=1}^\infty\lambda_j^{2\alpha}|\hat u_j|^2
\ee
[note that $D(A^0)$ coincides with $H$]. We can define $A^\alpha\:D(A^\alpha)\to H$ as the mapping
$$
\sum_{j=1}^\infty \hat u_jw_j\mapsto \sum_{j=1}^\infty \lambda_j^\alpha\hat u_jw_j,
$$
and then $\|u\|_{D(A^\alpha)}=\|A^\alpha u\|$. Note that $A^\alpha$ also makes sense as a mapping from $D(A^\beta)\to D(A^{\beta-\alpha})$ for any $\beta\in\R$, and that for $\beta\ge\alpha\ge0$ we have
\begin{equation}\label{higherpower}
D(A^\beta)=\{u\in D(A^{\beta-\alpha}):\ A^{\beta-\alpha}u\in D(A^\alpha)\}.
\end{equation}

We can define a semigroup $\e^{-\theta A}\:H\to H$ by setting
\begin{equation}\label{semigroup}
\e^{-\theta A}u :=\sum_{j=1}^\infty \e^{-\theta\lambda_j}\<u,w_j\>w_j,\qquad\theta\ge0;
\end{equation}
this extends naturally to $D(A^\alpha)$ for any $\alpha>0$, and for $\alpha<0$ we can interpret $\<u,w_j\>$ via the natural pairing between $D(A^\alpha)$ and $D(A^{-\alpha})$ (or, alternatively, as $\hat u_j$ in the definition \eqref{As-def}). Then for all $u\in D(A^\alpha)$ we have
\be{smoothing}
\|\e^{-\theta A}u\|_{D(A^\beta)}\le\begin{cases} C_{\beta-\alpha}\theta^{-(\beta-\alpha)}\|u\|_{D(A^\alpha)}& \beta\ge\alpha,\\
\e^{-\lambda_1\theta}\lambda_1^{\beta-\alpha}\|u\|_{D(A^\alpha)}&\beta<\alpha,
\end{cases}
\ee
where we can take $C_\gamma=\sup_{\lambda\ge0}\lambda^\gamma\e^{-\lambda}$ (the exact form of the constant is unimportant, but note $C_\gamma<\infty$ for every $\gamma\ge0$) and
\begin{equation} \label{hscontinuous}
\norm{e^{-{\theta}A}u - u}_{D(A^{\alpha})} \to 0 \quad \mbox{as} \quad {\theta} \to 0^+.
\end{equation}
In particular, \eqref{hscontinuous} means that $\e^{-\theta A}$ is a strongly continuous semigroup on $D(A^\alpha)$ for every $\alpha\in\R$.

Now suppose that we have a Banach space $\Lp$ such that
\begin{itemize}
\item[($\Lp$-i)] For some $\gamma_1\le\gamma_2$
\be{Lone}
D(A^{\gamma_2})\subset\Lp\subset D(A^{\gamma_1})
\ee
and
\item[($\Lp$-ii)] $\e^{-\theta A}$ is a uniformly bounded operator on $\Lp$ for $\theta\ge 0$, i.e.\ there exists a constant $C_\Lp>0$ such that
	\begin{equation} \label{bounded}
	\norm{\e^{-{\theta}A}u}_{\Lp} \leq C_{\Lp}\norm{u}_{\Lp} \quad \mbox{for} \quad {\theta} \geq 0,
	\end{equation}
	and $\e^{-\theta A}$ is a strongly continuous semigroup on $\Lp$, i.e.\ for each $u \in\Lp$
	\begin{equation} \label{lpcontinuous}
	\norm{\e^{-{\theta}A}u - u}_{\Lp} \to 0 \quad \mbox{as} \quad {\theta} \to 0^+.
	\end{equation}
\end{itemize}

We assume that the inclusions in ($\Lp$-i) are continuous so that, for example, $\Lp\subset D(A^{\gamma_1})$ means that we also have $\|u\|_{D(A^{\gamma_1})}\le C_{\Lp\to \gamma_1}\|u\|_{\Lp}$ for some constant $C_{\Lp\to\gamma_1}$ [there is an implicit abbreviation in the subscript, where we write $\gamma_1$ with $D(A^{\gamma_1})$].

Note that the embedding $\Lp\subset D(A^{\gamma_1})$ from \eqref{Lone} ensures that the definition of the semigroup in \eqref{semigroup} makes sense for $u\in\Lp$.

\sbsection{Approximation using the semigroup}

Using the semigroup $\e^{-\theta A}$ we can easily approximate any $u\in D(A^\alpha)\cap\Lp$ in a `good way' in both $D(A^\alpha)$ and $\Lp$. The following lemma simply combines the facts above to make this more explicit.

\begin{lemma}\label{easy-abstract}
Suppose that {\rm($\Lp$-i)} and {\rm($\Lp$-ii)} hold. If $u\in D(A^\alpha)\cap\Lp$ for some $\alpha\in\R$ and $u_\theta:=\e^{-\theta A}u$ then
  \begin{itemize}
  \item[{\rm(i)}] $u_\theta\in D(A^\beta)$ for every $\beta\in\R$ when $\theta>0$;
  \item[{\rm(ii)}] $\|u_\theta\|_{D(A^\alpha)}\le\|u\|_{D(A^\alpha)}$ for all $\theta>0$;
  \item[{\rm(iii)}] $\|u_\theta\|_{\Lp}\le C_{\Lp}\|u\|_{\Lp}$ for all $\theta>0$; and
  \item[{\rm(iv)}] $u_\theta\to u$ in $\Lp$ and in $D(A^\alpha)$ as $\theta\to0^+$.
  \end{itemize}
\end{lemma}

Note that if $u\in\Lp$ and ($\Lp$-i) holds then we can always find a value of $\alpha\in\R$ so that $u\in D(A^{\alpha})\cap\Lp$: if we have \eqref{Lone} then $u\in\Lp\cap D(A^{\gamma_1})$. 
If we want to apply the lemma as stated assuming explicitly only that $u\in D(A^\alpha)$ then to ensure that we also have $u\in\Lp$ we need to have $\alpha\ge\gamma_2$. 
Nevertheless, we always have (i), (ii), and (iv) for $u\in D(A^\alpha)$ for any $\alpha\in\R$.

\begin{proof}
  Parts (i) and (ii) both follow from \eqref{smoothing}, (iii) is \eqref{bounded}, and (iv) combines \eqref{hscontinuous} and \eqref{lpcontinuous}.
\end{proof}

Use of the semigroup like this can provide a natural way to produce a smooth approximation that is well tailored to the particular problem under consideration; see Robinson \& Sadowski \cite{RobinsonSadowski} for one example in the context of the Navier--Stokes equations, namely a straightforward proof of local well-posedness in $L^2(\R^3)\cap L^3(\R^3)$.

\sbsection{Approximation using eigenspaces}

We now want to obtain a similar approximation result, but for a set of approximations that lie in finite-dimensional space spanned by eigenfunctions of an operator $A$ satisfying the conditions above. This is the key abstract result of this paper; as with Lemma \ref{easy-abstract} its use in applications relies on the explicit identification of the fractional power spaces of certain common operators that we will recall in Section \ref{explicit}.
	
The approximation operator $\Pi_\theta$ introduced in \eqref{Pitheta} is related to the Bochner--Riesz means
$$
S_N^\gamma u:=\sum_{n=1}^N\left(1-\frac{\lambda_n}{\lambda_N}\right)^\gamma \<u,w_n\>w_n,
$$
which satisfy $S_N^\gamma u\to u$ in $L^p$ as $N\to\infty$ provided that $\gamma$ is sufficiently large (see \cite{Babenko} or \cite{Davis-Chang}). One could view $\Pi_\theta$ as a Bochner--Riesz mean of `exponential order', the exponential factor in the definition allowing for a much simpler proof of convergence than for $S_N^\gamma$ and with one operator that works for every $L^p$.

\begin{proposition} \label{approxlemma}
Suppose that {\rm($\Lp$-i)} and {\rm($\Lp$-ii)} hold. For ${\theta} > 0$ set
\begin{equation} \Pi_{\theta}u := 
\sum_{\lambda_n < {\theta}^{-2}}{\e^{-\theta\lambda_n}\dual{u}{w_n}w_n}.\label{Pitheta}
\end{equation}
Then
\begin{itemize}
	\item[{\rm(i)}] the range of $\Pi_{\theta}$ is the linear span of a finite number of eigenfunctions of $A$, so in particular $\Pi_\theta u\in D(A^\alpha)$ for every $\alpha\in\R$, and
			
	\item[{\rm(ii)}] if $X = \Lp$  or  $D(A^\alpha)$ for any $\alpha \in \mathbb{R}$, then
	\begin{enumerate}
		\item[{\rm(a)}] $\Pi_{\theta}$ is a bounded operator on $X$, uniformly for $\theta>0$, and
				
		\item[{\rm(b)}] for any $u \in X$ we have $\Pi_{\theta}u \to u$ in $X$ as ${\theta} \to 0^+$.
	\end{enumerate}
\end{itemize}
\end{proposition}
	
	
\begin{proof}
Property (i) is immediate from the definition of $\Pi_\theta$.
		
For (ii) we start with an auxiliary estimate for $u \in D(A^{\beta})$, $\beta \le \alpha$. If
$$ u = \sum_{n = 1}^{\infty}{\dual{u}{w_n}w_n} $$
then for every ${\theta} > 0$ we have
\begin{align*}
\norm{\Pi_{\theta}u - \e^{-{\theta}A}u}^2_{D(A^{\alpha})} &
= \sum_{\lambda_n \geq {\theta}^{-2}} {\lambda_n^{2\alpha}\e^{-2\lambda_n{\theta}}\abs{\dual{u}{w_n}}^2} \\
&\leq \sum_{\lambda_n \geq {\theta}^{-2}}{\lambda_n^{2\alpha}\e^{-2{\lambda_n}^{1/2}}\abs{\dual{u}{w_n}}^2}\\
& \leq \sum_{\lambda_n \geq {\theta}^{-2}}{\lambda_n^{2(\alpha - \beta)}\e^{-2{\lambda_n}^{1/2}}\lambda_n^{2\beta}\abs{\dual{u}{w_n}}^2} \\
& \leq \left(\sup_{\lambda\ge\theta^{-2}}\lambda^{2(\alpha-\beta)}\e^{-2\lambda^{1/2}}\right)\norm{u}_{D(A^\beta)}^2. \end{align*}

If for each $\kappa\in\R$ we set
$$
\Phi(\theta,\kappa):=\sup_{\lambda\ge\theta^{-2}}\lambda^{\kappa}\e^{-\lambda^{1/2}}
$$
then we have
\begin{equation} \label{keyestimate}
\norm{\Pi_{\theta}u - \e^{-{\theta}A}u}_{D(A^\alpha)} \leq \Phi({\theta}, \alpha - \beta)\norm{u}_{D(A^\beta)} \quad \mbox{for} \quad \beta \le\alpha.
\end{equation}
Since
$$
\Phi(\theta,\kappa)
=\begin{cases}
\theta^{-2\kappa}\e^{-1/\theta}&\kappa<0\mbox{ or }\kappa\ge 0,\ \theta\le (2\kappa)^{-1},\\
(2\kappa)^{2\kappa}\e^{-2\kappa}
&\kappa\ge0,\ \theta>(2\kappa)^{-1},
\end{cases}
$$
we have $\Phi(\theta,\kappa)\le M_\kappa$ for every $\theta>0$ and
\begin{equation}\label{Phi}
\Phi({\theta},\kappa) \to 0 \quad \mbox{as} \quad {\theta} \to 0^+\quad\mbox{for every }\kappa\ge0.
\end{equation}
		
It is immediate that $\Pi_{\theta}$ is bounded on $D(A^\alpha)$ given that $\Pi_{\theta}$ only decreases the modulus of the Fourier coefficients:
$$ \norm{\Pi_{\theta}u}_{D(A^{\alpha})} \leq \norm{u}_{D(A^{\alpha})}. $$
The convergence $\norm{\Pi_{\theta}u - u}_{D(A^\alpha)} \to 0$ as ${\theta} \to 0^+$, follows from \eqref{keyestimate} and \eqref{Phi} with $\beta=\alpha$ and the fact that $\e^{-{\theta}A}u \to u$ in $D(A^{\alpha})$ as ${\theta} \to 0^+$; we have
$$ \norm{\Pi_{\theta}u - u}_{D(A^\alpha)} \leq \norm{\Pi_{\theta}u - \e^{-{\theta}A}u}_{D(A^\alpha)} + \norm{\e^{-{\theta}A}u - u}_{D(A^\alpha)} \to 0 $$
as $\theta\to0^+$.

Now suppose that $u\in\Lp$. Since \eqref{keyestimate} shows that $\Pi_\theta u-\e^{-\theta A}u\in D(A^{\gamma_2})$ whenever $u\in D(A^{\gamma_1})$, we have
\begin{align*} \nonumber
\norm{\Pi_{\theta}u}_{{\Lp}} &= \norm{(\Pi_{\theta}u - \e^{-{\theta}A}u) + \e^{-{\theta}A}u}_{{\Lp}} \\& \leq C_{\gamma_2\to {\Lp}}\norm{\Pi_{\theta}u - \e^{-{\theta}A}u}_{D(A^{\gamma_2})} + \norm{\e^{-{\theta}A}u}_{{\Lp}} \\
&\leq C_{\gamma_2\to {\Lp}}\Phi({\theta}, \gamma_2-\gamma_1)\norm{u}_{D(A^{\gamma_1})} + C_{\Lp}\norm{u}_{{\Lp}}\\
&\leq \big[C_{\gamma_2\to {\Lp}}C_{{\Lp}\to\gamma_1}\Phi(\theta,\gamma_2-\gamma_1)+C_{\Lp}\big]\|u\|_{{\Lp}},
\end{align*}
using \eqref{Lone}, \eqref{bounded}, and (\ref{keyestimate}). It follows, since $\Phi(\theta,\gamma)\le M_\gamma$ independent of $\theta$, that
$$
\norm{\Pi_{\theta}u}_{{\Lp}} \leq K_{\Lp}\norm{u}_{{\Lp}},
$$
so $\Pi_{\theta}\: {\Lp} \to {\Lp}$ is bounded. Convergence of $\Pi_{\theta}u$ to $u$ as ${\theta} \to 0^+$ follows similarly, since
\begin{align*} \nonumber
\norm{\Pi_{\theta}u - u}_{{\Lp}} &\leq \norm{\Pi_{\theta}u - \e^{-{\theta}A}u}_{{\Lp}} + \norm{\e^{-{\theta}A}u - u}_{{\Lp}} \\
&\leq C_{\gamma_2\to {\Lp}}\norm{\Pi_{\theta}u - \e^{-{\theta}A}u}_{D(A^{\gamma_2})} + \norm{\e^{-{\theta}A}u - u}_{{\Lp}} \\
&\leq C_{\gamma_2\to {\Lp}}\Phi({\theta}, \gamma_2-\gamma_1)\norm{u}_{D(A^{\gamma_1})} + \norm{\e^{-{\theta}A}u - u}_{{\Lp}}\\
&\leq C_{\gamma_2\to\Lp}C_{\Lp\to\gamma_1}\Phi(\theta,\gamma_2-\gamma_1)\|u\|_{\Lp}+\|e^{-\theta A}u-u\|_{\Lp}
\end{align*}
and both terms tend to zero as ${\theta} \to 0^+$.
\end{proof}

\subsection{Further results via interpolation}

We note here for use later that it is possible to obtain additional results from either Lemma \ref{easy-abstract} or Proposition \ref{approxlemma} via interpolation. If $Y\subset X$ and $\Pi_\theta$ is a linear bounded operator on both $X$ and $Y$ then $\Pi_\theta$ is bounded on any 
(real or complex) interpolation space $Z:=[X,Y]_\alpha$.

Now suppose  in addition that $\|\Pi_\theta u-u\|_Y\to0$ as $\theta\to0^+$ for every $u\in Y$. Then, since $Y$ is dense in $Z$ (see \cite[Theorem 3.4.2]{Bergh-Lofstrom} for real interpolation, and \cite[Theorem 4.2.2]{Bergh-Lofstrom} for complex interpolation), we can show that
$$\|\Pi_\theta u-u\|_Z\to0\qquad\mbox{as}\qquad\theta\to0^+\qquad\mbox{for every }u\in Z.$$
Take $u\in Z$ and $v\in Y$, then
\begin{align*}
\|\Pi_\theta u-u\|_Z&=\|\Pi_\theta u-\Pi_\theta v+\Pi_\theta v-v+v-u\|_Z\\
&\le C\|u-v\|_Z+\|\Pi_\theta v-v\|_Z+\|u-v\|_Z;
\end{align*}
given $\eps>0$ choose $v\in Y$ such that $\|u-v\|_Z<\eps/2(1+C)$ and then $\theta$ small enough that
$$\|\Pi_\theta v-v\|_Z\le C\|\Pi_\theta v-v\|_X^{1-\alpha}\|\Pi_\theta v-v\|_Y^\alpha<\eps/2.
$$

\section{Fractional power spaces of the Lapalacian and Stokes operators}\label{explicit}

In this section we recall the explicit characterisation of the fractional power spaces of the negative Dirichlet Laplacian and Stokes operator on a sufficiently smooth bounded domain $\Omega$.

\begin{theorem}\label{fp-Dirichlet}
When $A$ is the negative Dirichlet Laplacian on $\Omega\subset\R^d$, $d\ge2$, we have
  $$
  D(A^\theta)=\begin{cases}
  H^{2\theta}(\Omega), &0<\theta< 1/4, \\
  H^{1/2}_{00}(\Omega),&\theta=1/4,\\
  H^{2\theta}_0(\Omega), &1/4<\theta\le 1/2, \\
  H^{2\theta}(\Omega)\cap H_0^1(\Omega), &1/2<\theta\le 1,
  \end{cases}
  $$
  where $H_{00}^{1/2}(\Omega)$ consists of all $u\in H^{1/2}(\Omega)$ such that
    $$
    \int_\Omega\rho(x)^{-1}|u(x)|^2\,\d x<\infty,
    $$
    with $\rho(x)$ any $C^\infty$ function comparable to $\dist(x,\partial\Omega)$. If $A$ is the Stokes operator on $\Omega$ with Dirichlet boundary conditions then the domains of the fractional powers of $A$ are as above, except that all spaces are intersected with
$$
H_\sigma:=\mbox{completion of }\{\phi\in [C_c^\infty(\Omega)]^d:\ \nabla\cdot\phi=0\}\mbox{ in the norm of }L^2(\Omega).
$$
\end{theorem}

The characterisation of the domains of the Dirichlet Laplacian can be found in the papers by Grisvard \cite{Grisvard}, Fujiwara \cite{Fujiwara67}, and Seeley \cite{Seeley3}. Note that  Fujiwara's statement is not correct for $\theta=3/4$, and that  Seeley also gives the corresponding characterisation for the operators in $L^p$-based spaces. For the Stokes operator $\A$, Giga \cite{Giga85} and Fujita \& Morimoto \cite{Fujita-Morimoto} both show that $D(\A)=D(A)\cap H_\sigma$; the former in the greater generality of $L^p$-based spaces. 

To guarantee that our approximating functions are smooth we will also need to consider $D(A^\theta)$ for $\theta>1$; here an inclusion will be sufficient.

\begin{corollary}\label{higher-power}
  If $A$ is the negative Dirichlet Laplacian on $\Omega$ then for $\theta\ge 1$
$$
  D(A^\theta)\subset H^{2\theta}\cap H_0^1,\qquad\mbox{with}\qquad \|u\|_{H^{2\theta}}\le C_{D(A^\theta)\to H^{2\theta}}\|A^\theta u\|
$$
  for every $u\in D(A^\theta)$.
  \end{corollary}

\begin{proof}
First we note that $D(A^\theta)\subseteq D(A)=H^2(\Omega)\cap H^1_0(\Omega)$ for every $\theta\ge1$; in particular $D(A^\theta)\subset H_0^1(\Omega)$, so we only need to show that
\be{theresult}
D(A^\theta)\subset H^{2\theta}(\Omega),\qquad\mbox{with}\qquad\|u\|_{H^{2\theta}}\le C_{D(A^\theta)\to H^{2\theta}}\|A^\theta u\|
\ee
for every $u\in D(A^\theta)$. Theorem \ref{fp-Dirichlet} shows that this holds for all $0<\theta\le1$.

  We now use \eqref{higherpower} and induction. Suppose that \eqref{theresult} holds for all $0<\theta\le k$ for some $k\in\N$; then for $\alpha=k+r$ with $0<r\le1$ we have
  \begin{align*}
  D(A^\alpha)&=D(A^{k+r})\\
  &=\{u:\ Au\in D(A^{k-1+r})\}\\
  &=\{u:\ -\Delta u\in D(A^{k-1+r})\},
  \end{align*}
  noting that since $u\in D(A^\alpha)$ and $\alpha\ge1$ we have $u\in D(A)$, which guarantees that $Au=-\Delta u$.

  It follows that any $u\in D(A^\alpha)$ solves the Dirichlet problem
  \be{elliptic}
  -\Delta u=f,\qquad u|_{\partial\Omega}=0,
  \ee
  for some $f\in D(A^{k-1+r})\subset H^{2(k-1+r)}(\Omega)$ using our inductive hypothesis. Elliptic regularity results for \eqref{elliptic} (see \cite[Theorem II.5.4]{LionsMagenes1}, for example) now guarantee that $u\in H^{2(k+r)}(\Omega)$ with
  \begin{align*}
  \|u\|_{H^{2(k+r)}}&\le c\|f\|_{H^{2(k-1+r)}}=c\|\Delta u\|_{H^{2(k-1+r)}}\\
  &=c\|Au\|_{H^{2(k-1+r)}}\le c\|A^{k+r}u\|,
  \end{align*}
  thanks to our inductive hypothesis.
\end{proof}

\section{Simultaneous approximation in Lebesgue and Sobolev spaces}\label{ourapprox}

  We can now combine the abstract approximation results from Section \ref{approx-abstract} with the characterisation of fractional power spaces from the previous section to give some more explicit approximation results. In all that follows we let $\Omega$ be a smooth bounded domain in $\R^n$, and by `smooth function on $\Omega$' we mean that a function is an element of $C^\infty(\overline{\Omega})$.

  \sbsection{Approximation respecting Dirichlet boundary conditions}

   In the abstract setting of Section \ref{approx-abstract} we take $H=L^2(\Omega)$, we let $A=-\Delta$, where $\Delta$ is the Laplacian on $\Omega$ with Dirichlet boundary conditions, and we take $\Lp=L^p(\Omega)$ for some $p\in(1,\infty)$ with $p\neq2$.

   We need to check the assumptions ($\Lp$-i) and ($\Lp$-ii) from Section \ref{sb-XpDr} on the relationship between the spaces $\Lp$ and $D(A^\alpha)$.

  \begin{itemize}
  \item[($\Lp$-i)] If we take $\Lp=L^p(\Omega)$ with $p\in(2,\infty)$ then since we are on a bounded domain, we have
  $$
  \Lp=L^p(\Omega)\subset L^2(\Omega)
  $$
and we can choose $\gamma\ge n(p-2)/4p$ so that
       $$
       D(A^\gamma)\subset H^{2\gamma}(\Omega)\subset L^p(\Omega)=\Lp.
       $$
       In this case \eqref{Lone} holds. If $\Lp=L^q(\Omega)$ for some $1<q<2$ we have $L^2(\Omega)\subset L^q(\Omega)$, and since $L^q(\Omega)$ is the dual space of some $L^p(\Omega)$ with $p>2$ we have
       $$
       L^q\simeq(L^p)^*\subset D(A^\gamma)^*=D(A^{-\gamma}),
       $$
       where $\gamma\ge n(2-q)/4q$.
       \item[($\Lp$-ii)] That $\e^{-\theta A}$ is bounded on $L^p(\Omega)$ for each $1<p<\infty$ follows from the analysis in Section 7.3 of Pazy \cite{Pazy}, as does the fact that $\e^{-\theta A}$ is a strongly continuous semigroup on $L^p(\Omega)$.
  \end{itemize}

   Our first approximation result uses the semigroup arising from the Dirichlet Laplacian, and is a corollary of Lemma \ref{easy-abstract}.

\begin{theorem}\label{approx-semi}
  If $u\in L^2(\Omega)$ then, for every $\theta>0$, $u_\theta:=\e^{-\theta A}u$ is smooth and zero on $\partial\Omega$. If in addition $u\in X$ then
  $$
  \|u_\theta\|_X\le C_X\|u\|_X,\qquad\mbox{and}\qquad\|u_\theta-u\|_X\to0\mbox{ as }\theta\to0^+,
  $$
  where we can take $X$ to be $H^s(\Omega)$ for $0<s<1/2$, $H^{1/2}_{00}(\Omega)$, $H^s_0(\Omega)$ for $1/2<s\le1$, $H^s(\Omega)\cap H_0^1(\Omega)$ for $1<s\le 2$, or $L^p(\Omega)$ for any $p\in(1,\infty)$.
\end{theorem}

\begin{proof}
  By part (i) of Lemma \ref{easy-abstract} we have $u_\theta\in D(A^r)$ for every $r\ge0$.  In particular $u_\theta\in D(A)=H^2\cap H_0^1$, so $u_\theta$ is zero on $\partial\Omega$. Since $D(A^r)\subset H^{2r}(\Omega)$ (Corollary \ref{higher-power}) it also follows that $u_\theta\in C^\infty(\overline{\Omega})$.

  The boundedness in Sobolev spaces follows from part (ii) of Lemma \ref{easy-abstract} using the characterisation of $D(A^\alpha)$ in Theorem \ref{fp-Dirichlet}, and the convergence in Sobolev spaces from part (iv) with $X=D(A^\alpha)$. The boundedness and convergence in $L^p$ follows from parts (iii) and (iv) of the same lemma.
\end{proof}

Proposition \ref{approxlemma} yields a corresponding result on approximation that combines the semigroup with a truncated eigenfunction expansion.

\begin{theorem}\label{approx-CF}
   Let $(w_j)$ denote the $L^2$-orthonormal eigenfunctions of the Dirichlet Laplacian on $\Omega$ with corresponding eigenvalues $(\lambda_j)$, ordered so that $\lambda_{j+1}\ge\lambda_j$. For any $u\in L^2(\Omega)$ set
  \begin{equation}\label{Pitheta2}
  u_\theta:=\Pi_\theta u= \sum_{\lambda_n < {\theta}^{-2}}\e^{-\theta\lambda_n}\<u,w_n\>w_n.
  \end{equation}
  Then $u_\theta$ has all the properties given in Theorem \ref{approx-semi}, and lies in the linear span of a finite number of eigenfunctions of $A$ for every $\theta>0$.
\end{theorem}

\sbsection{Approximation respecting Dirichlet boundary data and zero divergence}

  To deal with functions that have zero divergence we take $\A$ to be the Stokes operator, and set $H=L^2_\sigma(\Omega)$ and $\Lp=L^p_\sigma(\Omega)$  for some $p\in(1,\infty)$, $p\neq2$, where
    $$
    L^p_\sigma(\Omega):=\mbox{completion of }\{\phi\in C_c^\infty(\Omega):\ \nabla\cdot\phi=0\}\mbox{ in the }L^p(\Omega)\mbox{-norm}.
    $$
         Property ($\Lp$-i) from Section \ref{sb-XpDr} is checked as before, using the facts that $(L^p_\sigma)^*\simeq L^q_\sigma$ when $(p,q)$ are conjugate (see Theorem 2 part (2) in Fujiwara \& Morimoto \cite{FujiwaraMorimoto}) and that we have a continuous inclusion $D(\A^\gamma)\subset D(A^\gamma)$. The properties in ($\Lp$-ii) for the semigroup $\e^{-\A t}$ on $L^p_\sigma(\Omega)$ can be found as Theorem 2.1 in Miyakawa \cite{Miyakawa} or Giga \cite{Giga81a}.

 \begin{theorem}\label{approx-Stokes}
 Assume that $\Omega\subset\R^d$ with $d\le 4$. Take $u\in L^2(\Omega)$ and for every $\theta>0$ let
  $$
  u_\theta:=\e^{-\theta \A}u\qquad\mbox{or}\qquad u_\theta:=\Pi_\theta u,
  $$
  where $\Pi_\theta$ is defined as in \eqref{Pitheta2}, but now $(w_j)$ are the eigenfunctions of $\A$. Then $u_\theta$ is smooth, zero on $\partial\Omega$, and divergence free. If in addition $u\in X$ then
  $$
  \|u_\theta\|_X\le C_X\|u\|_X,\qquad\mbox{and}\qquad\|u_\theta-u\|_X\to0\mbox{ as }\theta\to0^+,
  $$
  where we can take $X$ to be $H^s(\Omega)\cap L^2_\sigma(\Omega)$ for $0<s<1/2$, $H^{1/2}_{00}(\Omega)\cap L^2_\sigma(\Omega)$, $H_0^s(\Omega)\cap L^2_\sigma(\Omega)$ for $1/2<s\le 1$, $H^s(\Omega)\cap H_0^1(\Omega)\cap L^2_\sigma(\Omega)$ for $1<s\le 2$, or $L^p_\sigma(\Omega)$ for any $p\in(1,\infty)$.
\end{theorem}

As before, this result follows by combining Lemma \ref{easy-abstract}, Proposition \ref{approxlemma}, and the identification of the fractional power spaces of the Stokes operator in Theorem \ref{fp-Dirichlet}. The restriction to $d\le 4$ is to ensure that $D(\A)\subset H^2\subset L^p$ for every $p\in(1,\infty)$. Without restriction on the dimension we then have to restrict to $1<p\le 2d/(d-4)$.

\subsection{Complex interpolation and approximation in $W^{k,p}(\Omega)$}

We can also use Lemma \ref{easy-abstract} or Proposition \ref{approxlemma} to obtain approximation results in $W^{k,p}(\Omega)$ by interpolation, provided we can verify that the approximation holds in the `endpoint' spaces $L^p(\Omega)$ and $D(A_p)$, where $A_p$ is the $L^p$-Laplacian. We restrict to dimension $d\le 4$ for simplicity.

From Theorem \ref{approx-CF} (Laplacian case) we know that for any $u\in L^p$
$$
\|\Pi_\theta u\|_{L^p}\le C\|u\|_{L^p}\qquad\mbox{and}\qquad\|\Pi_\theta u-u\|_{L^p}\to0\mbox{ as }\theta\to0^+.
$$
The domain of the $L^p$-Laplacian is $D(A_p):=W^{2,p}(\Omega)\cap W^{1,p}_0(\Omega)$. For $p>2$
$$
H^{2+d\left(\frac{1}{2}-\frac{1}{p}\right)}(\Omega)\cap H_0^1(\Omega)\subset W^{2,p}(\Omega)\cap W^{1,p}_0(\Omega)\subset H^2(\Omega)\cap H_0^1(\Omega)
$$
so
$$
D(A^{1+\frac{d}{2}\left(\frac{1}{2}-\frac{1}{p}\right)})\subset D(A_p)\subset D(A);
$$
while for $p<2$
$$
H^2(\Omega)\cap H^1_0(\Omega)\subset W^{2,p}\cap W_0^{1,p}\subset H^{2-d\left(\frac{1}{p}-\frac{1}{2}\right)}(\Omega)\cap H_0^1(\Omega)
$$
so
$$
D(A)\subset D(A_p)\subset D(A^{1-\frac{d}{2}\left(\frac{1}{p}-\frac{1}{2}\right)}).
$$
These give ($\Lp$-i) for the case $\Lp=D(A_p)$, and ($\Lp$-ii) follows easily from the fact that the heat semigroup is continuous in $L^p$: given any $u\in D(A_p)$, we have
$$
\|\e^{-\theta A_p}A_pu-A_pu\|_{L^p}=\|A_p(\e^{-\theta A_p}u-u)\|_{L^p}\to0
$$
(since $\e^{-\theta A_p}$ and $A_p$ commute) and the norm in $D(A_p)$ is the graph norm [$\|u\|_{L^p}+\|A_pu\|_{L^p}$].
Hence, from Proposition \ref{approxlemma}, for all $u\in D(A_p)$ we have
$$
\|\Pi_\theta u\|_{D(A_p)}\le C\|u\|_{D(A_p)}\qquad\mbox{and}\qquad\|\Pi_\theta u-u\|_{D(A_p)}\to0\mbox{ as }\theta\to0^+.
$$

Since the linear operator $\Pi_\theta$ is bounded on $L^p$ and $D(A_p)$, for any interpolation space $X_\alpha:=[L^p,D(A_p)]_\alpha$ with norm $\|\cdot\|_\alpha$ we also have
\be{unifbdnew}
\|\Pi_\theta u\|_\alpha\le C\|u\|_\alpha,
\ee
where $C$ can be chosen uniformly for all $\theta>0$.

Since $D(A_p)$ is dense in $X_\alpha$ \cite[Theorem 4.2.2]{Bergh-Lofstrom} and $\Pi_\theta$ is uniformly bounded on $X_\alpha$ as in \eqref{unifbdnew}, we can guarantee convergence of $\Pi_\theta u$ to $u$ in $X_\alpha$: given $u\in X_\alpha$ and $v\in D(A_p)$ we can write
\begin{align*}
\|\Pi_\theta u-u\|_\alpha&=\|\Pi_\theta u-\Pi_\theta v+\Pi_\theta v-v+v-u\|_\alpha\\
&\le C\|u-v\|_\alpha+\|\Pi_\theta v-v\|_\alpha+\|u-v\|_\alpha;
\end{align*}
given $\eps>0$ choose $v\in D(A_p)$ such that $\|u-v\|_\alpha<\eps/2(1+C)$ and then $\theta$ small enough that
$$\|\Pi_\theta v-v\|_\alpha\le C\|\Pi_\theta v-v\|_{L^p}^{1-\alpha}\|\Pi_\theta v-v\|_{D(A_p)}^\alpha<\eps/2.
$$

Identification of the interpolation spaces $X_\alpha$ is much more delicate in the non-Hilbertian case, and it is preferable to use complex interpolation methods. The generalisation of the results for the Laplacian to the case $p\neq2$ are given by Seeley \cite{Seeley3}:
$$
D(A_p^\alpha) = [L^p,D(A_p)]_\alpha=\begin{cases} W^{2\alpha,p}(\Omega),& 0 < \alpha<1/2p,\\
W_{00}^{1/p,p}(\Omega),&\alpha=1/2p,\\
W_0^{2\alpha,p}(\Omega),&1/2p<\alpha\le 1/2,\\
W^{2\alpha,p}(\Omega)\cap W^{1,p}_0(\Omega),&1/2<\alpha<1,
\end{cases}
$$
where $W_{00}^{1/p,p}(\Omega)$ consists of all $u\in W^{1/p,p}(\Omega)$ such that
$$
\int_\Omega\rho(x)^{-1}|u(x)|^p\,\d x<\infty,
$$
with $\rho$ as in the statement of Theorem \ref{fp-Dirichlet}. Results for the Stokes operator in $L^p$ can be found in Giga \cite{Giga85}.

\section{Application: the energy equality for the CBF equations}\label{sec-CBF}

In this section we will apply the eigenspace-approximation result of Theorem \ref{approx-Stokes} to  prove energy conservation for the 3D convective Brinkman--Forchheimer (CBF) equations
\begin{equation}\label{alphacrit}
\partial_tu -\mu\Delta u + (u \cdot \nabla)u + \nabla p + \beta|u|^ru = 0, \qquad\nabla\cdot u=0
\end{equation}
in the critical case $r=2$, when posed on a smooth bounded domain $\Omega\subset\R^3$ equipped with Dirichlet boundary conditions $u|_{\partial\Omega} = 0$. Here $u(x, t)\in\R^3$ is the velocity field and the scalar function $p(x, t)$ is the~pressure. The constant $\mu$ denotes the positive Brinkman coefficient (effective viscosity) and $\beta\ge0$ denotes the Forchheimer coefficient (proportional to the porosity of the material).

While these equations can be physically motivated, our interest in them here is primarily mathematical, as a version of the Navier--Stokes equations with an additional dissipative term $+\beta|u|^ru$. Unlike the Navier--Stokes equations themselves, for which known results are a long way from providing the global existence of regular solutions, for the CBF equations strong solutions
$$
u\in L^\infty(0,T;H_0^1)\cap L^2(0,T;H^2)
$$
are known to exist for all time \textit{for every }$r>2$ (Kalantarov \& Zelik \cite{KalantarovZelik}; see also Hajduk \& Robinson \cite{HR} for a simpler proof in the absence of boundaries and when $r=2$ and $4\mu\beta\ge1$).

We do not give full details of the argument that guarantees the validity of the energy equality for weak solutions, since it follows that in \cite{HR} extremely closely. Instead we define weak solutions precisely and then give a sketch of the proof, showing how Theorem \ref{approx-Stokes} allows the argument to be extended to the CBF equations on bounded domains.

\sbsection{Weak solutions of the CBF equations}

We use the standard notation for the vector-valued function spaces which often appear in the theory of fluid dynamics. For an arbitrary domain $\Omega \subseteq \mathbb{R}^n$ we define
$$
\mathcal{D}_{\sigma}(\Omega) := \lrb{\varphi \in C_{c}^{\infty}(\Omega) : \dv{\varphi} = 0}
$$
and
$$
H := \mbox{closure of} \ \mathcal{D}_{\sigma}(\Omega) \ \mbox{in}\ L^2(\Omega).
$$
The space of divergence-free test functions in the space-time domain is denoted by
$$
 \mathcal{D}_{\sigma}({\Omega_T}) := \lrb{\varphi \in C_{c}^{\infty}({\Omega_T}) : \dv{\varphi(\cdot, t)} = 0},
$$
where $\Omega_T := \Omega \times [0, T)$ for $T > 0$. Note that $\varphi(x, T) = 0$ for all $\varphi \in \mathcal{D}_{\sigma}({\Omega_T})$. We set $\mathcal{D}_\sigma(\Omega_\infty)=\cup_{T>0}\mathcal{D}_\sigma(\Omega_T)$.

We equip the space $H$ with the inner product induced by $L^2(\Omega)$; we denote it by $\dual{\cdot}{\cdot}$, and the corresponding norm by $\norm{\cdot}$.

We will use the following definition of a weak solution (cf.\ the corresponding definition of a weak solution for the Navier--Stokes equations in Robinson, Rodrigo, \& Sadowski \cite{RRSbook}).

\begin{definition} \label{defi:slabecbfr}
We will say that the function $u$ is \emph{a weak solution} on the time interval $[0, T)$ of the critical convective Brinkman--Forchheimer equations {\rm[}\eqref{alphacrit} with $r=2${\rm]} with the~initial condition $u_0 \in H$, if
$$ u \in L^{\infty}({0, T; H}) \cap L^{4}({0, T; L^4_\sigma}) \cap L^2({0, T; H_0^1}) $$
and
\begin{align}
-\int_{0}^{t}&\<u(s),\partial_t\varphi(s)\> + \mu\int_{0}^{t}\<\nabla u(s),\nabla \varphi(s)\> + \int_{0}^{t}\<(u(s) \cdot \nabla)u(s),\varphi(s)\>\nonumber \\
&+ \beta\int_{0}^{t}\<|u(s)|^2u(s),\varphi(s)\> = -\<u(t),\varphi(t)\> + \<u(0),\varphi(0)\>\label{weakform}
\end{align}
for almost every $t\in(0,T)$ and all test functions $\varphi \in \mathcal{D}_{\sigma}({\Omega_T})$.

A function $u$ is \textit{a global weak solution} if it is a weak solution on $[0,T)$ for every $T > 0$.
\end{definition}

Note that this definition coincides with the definition of a weak solution of the Navier--Stokes equations in the case $\beta=0$ if we drop the requirement that $u\in L^4(0,T;L^4)$.

Just as with the conventional Navier--Stokes equations, it is possible to replace the space of test functions $\mathcal{D}_{\sigma}$ in the weak formulation \eqref{weakform} with a number of other collections of functions. In order to use our eigenspace approximation for this model, we will want to replace $\mathcal{D}_\sigma$ with the space $\tilde{\mathcal{D}}_{\sigma}$ consisting of finite combinations of eigenfunctions of the Stokes operator. We therefore define
\begin{equation} \nonumber
\tilde{\mathcal{D}}_{\sigma}(\Omega_\infty) := \lrb{\varphi:\ \varphi = \sum_{k = 1}^{N}{\alpha_k(t)w_k(x)},\ \alpha_k \in C^1_0([0, \infty)), \ \mbox{for some } N \in \N},
\end{equation}
where $(w_k)$ are the eigenfunctions of the Stokes operator as in Theorem \ref{approx-Stokes}.

The functions in the space $\tilde{\mathcal{D}}_{\sigma}$ are less regular in time than those in $\mathcal{D}_{\sigma}$; they also do not have compact support within the spatial domain $\Omega$. However, they have the advantage that their dependence on the space and time variables is separated, and - crucial for our application here - that they are directly connected with the Stokes operator. We only state the following lemma here, since it follows that of Lemma 3.11 in Robinson et al.\ \cite{RRSbook} or Lemma 2.3 in Galdi \cite{Galdi} extremely closely.

\begin{lemma} \label{savesus}
If $u \in L^{\infty}(0, T; H) \cap L^4(0, T; L^4_{\sigma}) \cap L^{2}(0, T; H_0^1)$ for all $T > 0$, then $u$ satisfies \eqref{weakform} for every $\varphi \in \mathcal{D}_{\sigma}(\Omega_{\infty})$ iff it satisfies \eqref{weakform}  for every $\varphi \in \tilde{\mathcal{D}}_{\sigma}(\Omega_{\infty})$.
\end{lemma}

Weak solutions that satisfy the energy inequality exist for the CBF equations just as they do for the Navier--Stokes equations.

\begin{definition}
\textit{A Leray--Hopf weak solution} of the critical convective Brink\-man--Forchheimer equations \eqref{alphacrit} {\rm[$r=2$]} with the initial condition $u_0 \in H$ is a weak solution satisfying the following \emph{strong energy inequality}:
\begin{align} \label{energyineq}
\norm{u(t_1)}^2 + 2\mu\int_{t_0}^{t_1}{\norm{\nabla u(s)}^2 \, \mathrm{d}s} + 2\beta\int_{t_0}^{t_1}{\norm{u(s)}_{L^{4}(\Omega)}^{4} \, \mathrm{d}s} \leq \norm{u(t_0)}^2
\end{align}
for almost all initial times $t_0 \in [0, T)$, including zero, and all $t_1 \in (t_0, T)$.
\end{definition}

It is known that for every $u_0 \in H$ there exists at least one global Leray--Hopf weak solution of $(\ref{alphacrit})$, see Antontsev \& de Oliveira \cite{{Oliveira}}. A proof of the corresponding result for the 3D Navier--Stokes equations (i.e.\ \eqref{alphacrit} with $\beta = 0$) can be found in many places, e.g.\ in \cite{Galdi} or \cite{RRSbook}. However, it is not known if all weak solutions of the Navier--Stokes equations have to satisfy the energy inequality \eqref{energyineq} (with $\beta=0$). [The recent result of Buckmaster \& Vicol \cite{Buckmaster-Vicol} shows that solutions in the sense of distributions need not satisfy the energy inequality, thereby proving also the non-uniqueness of such solutions.] The problem of proving equality in \eqref{energyineq} for weak solutions of the Navier--Stokes equations is also open; there are only partial results in this direction, but it is known that the energy equality is satisfied by any weak solution $u\in L^4(0,T;L^4)$ (Serrin \cite{Serrin}). Since weak solutions of the CBF equations automatically satisfy this condition, one might expect that they satisfy the energy equality. This was shown by Hajduk \& Robinson in the periodic setting \cite{HR}; the purpose of this section is to show how the argument there can be adapted to the case of a smooth bounded domain by using the eigenspace-based approximation from Theorem \ref{approx-Stokes}.

\sbsection{Proof of the energy equality}

In this section we sketch a proof of the following theorem.

\begin{theorem} \label{twr:ee}
When $r=2$ every weak solution of \eqref{alphacrit} with initial condition $u_0 \in H$ satisfies the \emph{energy equality}:
$$
\norm{u(t_1)}^2 + 2\mu\int_{t_0}^{t_1}{\norm{\nabla u(s)}^2 \, \mathrm{d}s} + 2\beta\int_{t_0}^{t_1}{\norm{u(s)}_{{L}^{4}(\Omega)}^{4} \, \mathrm{d}s} = \norm{u(t_0)}^2
$$
for all $0 \leq t_0 < t_1 <T$. Hence, all weak solutions are continuous functions into the~phase space $L^2$, i.e.\ $u \in C(\lra{0, T}; H)$.
\end{theorem}

Note that to prove this result we require the more refined result of Proposition \ref{approxlemma}, which enables an approximation that uses only finite-dimensional eigenspaces of the Stokes operator. This approximation is not compactly supported but Lemma \ref{savesus} allows us to use it as a test function in the weak formulation \eqref{weakform}. The `approximation by semigroup' result of Lemma \ref{easy-abstract} is not sufficient since we do not have a version of Lemma \ref{savesus} for the functions arising from this kind of approximation.

\begin{proof} (Sketch)

We only sketch the proof, which follows that from Hajduk \& Robinson \cite{HR}, which in turn is based on the argument presented in Galdi \cite{Galdi}.

We approximate $u(t)$ for each $t \in [0, T]$ in such a way that
\begin{itemize}
\item[(i)] $u_n(t) \in \tilde{\D}_{\sigma}(\Omega)$,
\item[(ii)] $u_n(t) \to u(t)$ in $H_0^1(\Omega)$ with $\|u_n(t)\|_{H^1} \le C\|u(t)\|_{H^1}$,
\item[(iii)] $u_n(t) \to u(t)$ in $L^4(\Omega)$ with $\|u_n(t)\|_{L^4} \le C\|u(t)\|_{L^4}$, and
\item[(iv)] $u_n(t)$ is divergence free and zero on $\partial\Omega$,
\end{itemize}
with (ii)--(iv) holding for almost every $t \in [0, T]$. In (i) we want $u_n(t)$ to be in the finite-dimensional space spanned by the first $n$ eigenfunctions of the Stokes operator; we can obtain such an approximation using Theorem $\ref{approx-Stokes}$ by setting
$$ u_n(t) := \Pi_{1/n}u(t) = \sum_{\lambda_j < n^2}{\e^{-\lambda_j/n}\dual{u(t)}{w_j}w_j} $$
for each $t \in [0, T]$.

In the proof we will need the fact that
\be{L44c}
\|u_n-u\|_{L^4(0,T;L^4)}\to0\qquad\mbox{as}\qquad n\to\infty,
\ee
which follows from (iii): since $u\in L^4(0,T;L^4)$  and $\|u_n(t)-u(t)\|_{L^4}\to0$ for almost every $t\in[0,T]$ we can obtain \eqref{L44c} by an application of the Dominated Convergence Theorem (with dominating function $(1+C)\|u(t)\|_{L^4}$). A similar argument (using (ii)) shows that
$$
\|u_n-u\|_{L^2(0,T;H^1)}\to0\qquad\mbox{as}\qquad n\to\infty.
$$

To prove the energy equality for some time $t_1>0$ we set
$$
u_n^h(t):=\int_0^{t_1}\eta_h(t-s)u_n(s)\,\d s,
$$
where $\eta_h$ is an even mollifier. Since $u_n^h\in\tilde{\D}_\sigma(\Omega_T)$ we can use it as a test function in \eqref{weakform}:
\begin{align*}
-\int_{0}^{t}&\<u(s),\partial_tu_n^h(s)\> + \mu\int_{0}^{t}\<\nabla u(s),\nabla u_n^h(s)\> + \int_{0}^{t}\<(u(s) \cdot \nabla)u(s),u_n^h(s)\>\\
&+ \beta\int_{0}^{t}\<|u(s)|^2u(s),u_n^h(s)\> = -\<u(t),u_n^h(t)\> + \<u(0),u_n^h(0)\>.
\end{align*}

We first take the limit as $n\to\infty$. The limits in the linear terms are relatively straightforward. In the Navier--Stokes nonlinearity we can use
\begin{align*}
\Big|\int_{0}^{t_1}\<(u(s)& \cdot \nabla)u_n^h(s),u(s)\> \, \mathrm{d}s - \int_{0}^{t_1}\<(u(s) \cdot \nabla)u^h(s),u(s)\> \, \mathrm{d}s\Big| \\
&\leq \int_{0}^{t_1}\|u(s)\|_{L^4}^2\|\nabla u_n^h(s)-\nabla u^h(s)\|\,\d s\\
&\leq \norm{u}_{L^4({0, T; L^4})}^2\|u_n^h - u^h\|_{L^2({0, T; H_0^1})}.
\end{align*}
In the Forchheimer term $|u|^2u$ we have
\begin{align*}
\Big|\int_{0}^{t_1}\<\abs{u(s)}^2&u(s),u_n^h(s)\> \, \mathrm{d}s - \int_{0}^{t_1}{\dual{\abs{u(s)}^2u(s)}{u^h(s)} \, \mathrm{d}s}\Big| \\
&\leq \int_{0}^{t_1}{{\norm{u(s)}^3_{L^4}\|u_n^h(s) - u^h(s)\|_{L^4}}} \, \mathrm{d}s \\
&\leq \norm{u}^3_{L^4({0, T; L^4})}\|u_n^h - u^h\|_{L^4({0, T; L^4})} .
\end{align*}

By our choice of $u^h$ we have
$$
\int_0^{t_1}\<u,\partial_tu^h\>\,\d s=\int_0^{t_1}\int_0^{t_1}\dot\eta_h(t-s)\<u(t),u(s)\>\,\d t\,\d s=0
$$
and so
\begin{align*}
\mu\int_{0}^{t_1}\<\nabla u&,\nabla u^h\>+ \int_{0}^{t_1}\<(u \cdot \nabla)u,u^h\>+\beta\int_0^{t_1}\<|u|^2u,u^h\>\\
&= -\<u(t_1),u^h(t_1)\> + \<u(0),u^h(0)\>.
\end{align*}
Next we let $h\to0$, for which the argument is similar; we use the facts that the mollifier $\eta_h$ integrates to $1/2$ on the positive real axis and that $u$ is weakly continuous into $L^2$ to show that the right-hand side tends to
$$
-\frac{1}{2}\|u(t_1)\|^2+\frac{1}{2}\|u(t_0)\|^2.
$$

The continuity of $u$ into $L^2$ now follows by combining the weak continuity into $L^2$ and the continuity of $t\mapsto\|u(t)\|$, which is a consequence of the energy equality.
\end{proof}

\section{Conclusion}

Returning to the issues discussed in the introduction, recall that while the `spherical' truncation of a Fourier expansion
$$
u_n:=\sum_{|k|\le n}\hat u_k\e^{\ri k\cdot x}
$$
does not behave well in terms of boundedness/convergence in $L^p$ spaces, the `cubic' component-by-component truncation
$$
u_{[n]}:=\sum_{|k_j|\le n}\hat u_k\e^{\ri k\cdot x},\qquad k=(k_1,\ldots,k_d),
$$
does.

One can expect (cf.\ Babenko \cite{Babenko}) that there are similar problems in using a straightforward truncation of an expansion in terms of an orthonormal family of eigenfunctions:
$$
P_\lambda u:=\sum_{\lambda_n\le\lambda}\<u,w_n\>w_n,
$$
(where $Aw_n=\lambda_nw_n$). It is natural to ask if there is a `good' choice of eigenfunctions such that the truncations
$$
P_n u:=\sum_{w\in E_n}\<u,w\>w,
$$
where $E_n$ is some collection of eigenfunctions, is well-behaved with respect to the $L^p$ spaces. To our knowledge this is entirely open.

\bibliography{biblio}

\def\cprime{$'$} \def\cprime{$'$}
\begin{thebibliography}{10}

\bibitem{Oliveira}
{\sc Antontsev, S.~N., and de~Oliveira, H.~B.}
\newblock The {N}avier--{S}tokes problem modified by an absorption term.
\newblock {\em Appl. Anal. 89}, 12 (2010), 1805--1825.

\bibitem{Babenko}
{\sc Babenko, K.~I.}
\newblock The summability and convergence of the eigenfunction expansions of a
  differential operator.
\newblock {\em Mat. Sb. (N.S.) 91(133)\/} (1973), 147--201, 287.

\bibitem{Bergh-Lofstrom}
{\sc Bergh, J., and L\"{o}fstr\"{o}m, J.}
\newblock {\em Interpolation spaces. {A}n introduction}.
\newblock Grundlehren der Mathematischen Wissenschaften, No. 223.
  Springer-Verlag, Berlin-New York, 1976.

\bibitem{Buckmaster-Vicol}
{\sc Buckmaster, T., and Vicol, V.}
\newblock Nonuniqueness of weak solutions to the {N}avier-{S}tokes equation.
\newblock {\em Ann. of Math. (2) 189}, 1 (2019), 101--144.

\bibitem{Davis-Chang}
{\sc Davis, K.~M., and Chang, Y.-C.}
\newblock {\em Lectures on {B}ochner-{R}iesz means}, vol.~114 of {\em London
  Mathematical Society Lecture Note Series}.
\newblock Cambridge University Press, Cambridge, 1987.

\bibitem{CLF}
{\sc Fefferman, C.~L.}
\newblock The multiplier problem for the ball.
\newblock {\em Ann. of Math. (2) 94\/} (1971), 330--336.

\bibitem{Fujita-Morimoto}
{\sc Fujita, H., and Morimoto, H.}
\newblock On fractional powers of the {S}tokes operator.
\newblock {\em Proc. Japan Acad. 46\/} (1970), 1141--1143.

\bibitem{Fujiwara67}
{\sc Fujiwara, D.}
\newblock Concrete characterization of the domains of fractional powers of some
  elliptic differential operators of the second order.
\newblock {\em Proc. Japan Acad. 43\/} (1967), 82--86.

\bibitem{FujiwaraMorimoto}
{\sc Fujiwara, D., and Morimoto, H.}
\newblock An {$L_{r}$}-theorem of the {H}elmholtz decomposition of vector
  fields.
\newblock {\em J. Fac. Sci. Univ. Tokyo Sect. IA Math. 24}, 3 (1977), 685--700.

\bibitem{Galdi}
{\sc Galdi, G.~P.}
\newblock An introduction to the {N}avier--{S}tokes initial-boundary value
  problem.
\newblock In {\em Fundamental directions in mathematical fluid mechanics}, Adv.
  Math. Fluid Mech. Birkh\"auser, Basel, 2000, pp.~1--70.

\bibitem{Giga81a}
{\sc Giga, Y.}
\newblock Analyticity of the semigroup generated by the {S}tokes operator in
  {$L_{r}$}\ spaces.
\newblock {\em Math. Z. 178}, 3 (1981), 297--329.

\bibitem{Giga85}
{\sc Giga, Y.}
\newblock Domains of fractional powers of the {S}tokes operator in {$L_r$}
  spaces.
\newblock {\em Arch. Rational Mech. Anal. 89}, 3 (1985), 251--265.

\bibitem{Grafakos}
{\sc Grafakos, L.}
\newblock {\em Classical {F}ourier analysis}, third~ed., vol.~249 of {\em
  Graduate Texts in Mathematics}.
\newblock Springer, New York, 2014.

\bibitem{Grisvard}
{\sc Grisvard, P.}
\newblock Caract\'{e}risation de quelques espaces d'interpolation.
\newblock {\em Arch. Rational Mech. Anal. 25\/} (1967), 40--63.

\bibitem{HR}
{\sc Hajduk, K.~W., and Robinson, J.~C.}
\newblock Energy equality for the $3${D} critical convective
  {B}rinkman--{F}orchheimer equations.
\newblock {\em J. Diff. Eq. 263}, 11 (2017), 7141--7161.

\bibitem{KalantarovZelik}
{\sc Kalantarov, V., and Zelik, S.}
\newblock Smooth attractors for the {B}rinkman--{F}orchheimer equations with
  fast growing nonlinearities.
\newblock {\em Commun. Pure Appl. Anal. 11}, 5 (2012), 2037--2054.

\bibitem{LionsMagenes1}
{\sc Lions, J.-L., and Magenes, E.}
\newblock {\em Non-homogeneous boundary value problems and applications. {V}ol.
  {I}}.
\newblock Springer-Verlag, New York-Heidelberg, 1972.
\newblock Translated from the French by P. Kenneth, Die Grundlehren der
  mathematischen Wissenschaften, Band 181.

\bibitem{Miyakawa}
{\sc Miyakawa, T.}
\newblock On the initial value problem for the {N}avier-{S}tokes equations in
  {$L^{p}$}\ spaces.
\newblock {\em Hiroshima Math. J. 11}, 1 (1981), 9--20.

\bibitem{Muscalu-Schlag}
{\sc Muscalu, C., and Schlag, W.}
\newblock {\em Classical and multilinear harmonic analysis. {V}ol. {I}},
  vol.~137 of {\em Cambridge Studies in Advanced Mathematics}.
\newblock Cambridge University Press, Cambridge, 2013.

\bibitem{Pazy}
{\sc Pazy, A.}
\newblock {\em Semigroups of linear operators and applications to partial
  differential equations}, vol.~44 of {\em Applied Mathematical Sciences}.
\newblock Springer-Verlag, New York, 1983.

\bibitem{RRSbook}
{\sc Robinson, J.~C., Rodrigo, J.~L., and Sadowski, W.}
\newblock {\em The three-dimensional {N}avier--{S}tokes equations, classical
  theory}.
\newblock Cambridge Studies in Advanced Mathematics. Cambridge University
  Press, Cambridge, 2016.

\bibitem{RobinsonSadowski}
{\sc Robinson, J.~C., and Sadowski, W.}
\newblock A local smoothness criterion for solutions of the 3{D}
  {N}avier--{S}tokes equations.
\newblock {\em Rend. Semin. Mat. Univ. Padova 131\/} (2014), 159--178.

\bibitem{Seeley3}
{\sc Seeley, R.}
\newblock Interpolation in {$L^{p}$} with boundary conditions.
\newblock {\em Studia Math. 44\/} (1972), 47--60.

\bibitem{Serrin}
{\sc Serrin, J.}
\newblock The initial value problem for the {N}avier--{S}tokes equations.
\newblock In {\em Nonlinear {P}roblems ({P}roc. {S}ympos., {M}adison, {W}is.,
  1962)}. Univ. of Wisconsin Press, Madison, Wis., 1963, pp.~69--98.

\end{thebibliography}
\bibliographystyle{acm}
\end{document}